\documentclass[11pt]{article}
\usepackage{amsmath,amsfonts,amssymb,amsthm,thmtools,thm-restate,mathabx,stmaryrd,enumerate,ytableau,hyperref,multirow}
\usepackage[boxed,linesnumbered]{algorithm2e}
\usepackage{cleveref}
\usepackage{fullpage}
\usepackage{color}
\usepackage{authblk}
\usepackage{float}
\usepackage{tikz}
\usepackage{ytableau}
\newif\ifdraft
\drafttrue
\DeclareFontFamily{OMX}{MnSymbolE}{}
\DeclareFontShape{OMX}{MnSymbolE}{m}{n}{
    <-6>  MnSymbolE5
   <6-7>  MnSymbolE6
   <7-8>  MnSymbolE7
   <8-9>  MnSymbolE8
   <9-10> MnSymbolE9
  <10-12> MnSymbolE10
  <12->   MnSymbolE12}{}
\DeclareSymbolFont{mnlargesymbols}{OMX}{MnSymbolE}{m}{n}
\SetSymbolFont{mnlargesymbols}{bold}{OMX}{MnSymbolE}{b}{n}
\DeclareMathDelimiter{\llangle}{\mathopen}{mnlargesymbols}{'164}{mnlargesymbols}{'164}
\DeclareMathDelimiter{\rrangle}{\mathclose}{mnlargesymbols}{'171}{mnlargesymbols}{'171}

\declaretheorem[numberwithin=section]{theorem}

\title{Linear Adjusting Programming in Factor Space}
\author{Jing He\textsuperscript{\rm 1}\thanks{Swinburne University of Technology, Melbourne, 3122.Australia.Email:jinghe@swin.edu.au} Qi-Wei Kong\textsuperscript{\rm 2},  Ho-Chung Lui\textsuperscript{\rm 3}, Hai-Tao Liu\textsuperscript{\rm 3} \\  Yi-Mu Ji\textsuperscript{\rm 4} Hai-Chang Yao\textsuperscript{\rm 4}, Mo-Zhengfu Liu\textsuperscript}
\affil[1]{School of Software and Electrical Engineering, Swinburne University of Technology, Hawthorn, Australia}
\affil[2]{Institute of Information Engineering, Nanjing University of Finance and Economics, Nanjing, China}
\affil[3]{Research Center on Fictitious economy \& Data Science, Chinese Academy of Sciences,  Beijing, China}
\affil[4]{School of Computer Science, Nanjing University of Posts and Telecommunications, Nanjing, China}
\date{}
\begin{document}
\maketitle
\begin{abstract}
The definition of factor space and a unified optimization based classification model were developed for linear programming and supervised learning. Intelligent behaviour appeared in a decision process can be treated as a moving point \textit{y}, the dynamic state observed and controlled by the agent, moving in a factor space impelled by the goal factor and blocked by the constraint factors. Suppose that the feasible region is cut by a group of hyperplanes, when point \textit{y} reaches the region’s boundary, a hyperplane will block the moving and the agent needs to adjust the moving direction such that the target is pursued as faithful as possible. Since the boundary is not able to be represented to a differentiable function, the gradient method cannot be applied to describe the adjusting process. We, therefore, suggest a new model, named linear adjusting programming (LAP) in this paper. LAP is similar as a kind of relaxed linear programming (LP), and the difference between LP and LAP is: the former aims to find out the ultimate optimal point, while the latter just does a direct action in short period. You may ask: Where will a blocker encounter? How can the moving direction be adjusted? Where further blockers may be encountered next, and how should the direction be adjusted again? If the ultimate best is found, it is a blessing; if not, it is fine to find a local optimal solution. We request at least an adjusting should be achieved at the first time. what are the former and latter? possible to be more exact? In place of gradient vector, the projection of goal direction \textit{g} in a subspace plays a core role in linear adjusting programming. If a hyperplane blocks \textit{y} going ahead along with the direction  \textit{d}, then we must adjust the new direction \textit{d}’ as the projection of \textit{g} in the blocking plane. If there is only one blocker at a time, it is straightforward to calculate the projection, but how to calculate the projection when there are more than one blocker encountered simultaneously? It is an open problem for LP researchers still (M. Hassan, M. Rehmani, and J. Chen 2019)\cite{Chen2019} ( P. Wang. 2020)\cite{Wang2020}. We suggest a projection calculation by means of the Hat matrix in this paper. Linear adjusting programming will attract interest in economic restructuring, financial prediction, and reinforcement learning. It might bring a new light to solve the linear programming problem with a strong polynomial solution.
\end{abstract}

\section{Introduction} \label{sec:introduction}

Factor space is a mathematical theory proposed in 1982 (P. Z. Wang, M. Sugeno 1982)\cite{Wang1982}, which provides a unified mathematical framework for artificial intelligence and data science. All classifications and decisions are made in factor space. Based on that, a unified classification model was refined in 2020. An intelligent decision process can be described as a point \emph{y}, the dynamic state of an agent (or, observed and controlled by the agent). The agent's target is represented by a vector \emph{g}, which expels the movement of the point \emph{y}. In practice, the vector \emph{g} can be the economic restructuring or management policy. Apart from the target, there are hyperplanes present the contraint conditions for \emph{y}; each one cuts off half of the space, and they form a feasible convex region. The wall of a feasible region will block the moving trace when \emph{y} reaches the wall, and the agent needs to adjust its goal direction of \emph{y}. If a hyperplane $\alpha:(\tau, \emph{y}) =\emph{c}$ blocks  \    \emph{y}, then we must adjust the new direction\ \emph{d}' as the projection of \ \emph{g}:
\renewcommand{\theequation}{\thesection.\arabic{equation}}
\setcounter{equation}{0}
\begin{equation}
\emph{d'} = \emph{g}\;{\downarrow _\alpha} = \;\textit{g} - \tau* (\emph{g},\tau)/(\tau,\tau)
\label{AAA}
\end{equation}

But how do you calculate the projection when there are more than one blocker encountered in the same stage point? It is still an open problem for LP researchers (X. Cai.2020)\cite{Cai2020} (Y. Wang. 2019)\cite{YWang2019}.

Since there is no differentiable function to represent the wall, the classical gradient method cannot be applied here, and a new model is proposed in next section. It will be of great significance to economic restructuring, financial prediction, machine learning and reinforcement learning.

Authors were engaged in the research of linear programming. Simplex presented by Dantzig (G. B. Dantzig 2002)\cite{Dantzig2002} is a piece of perfect mathematical art; the only defect is that may rotate along edges. But how can the optimization be fastened? We need geometric description for simplex. The cone-cutting theory was proposed in 2011(P. Z. Wang 2011)\cite{Wang2011}, which  intuitively shows that a pivoting performed in the standard simplex tableau is taking a cone-cutting in the dual space. Utilizing the idea of cone-cutting, an algorithm was put forward in 2014 (P. Z. Wang 2014)\cite{Wang2014}, named Gradient Falling, which searching the minimum point as a body falling by gravity in dual space. The critical problem is how to calculate the projection of a vector in subspace, and a projection calculation algorithm was given by the author himself. In 2017, Lui improved the algorithm by citing null projection method from Matlab, renamed Gravity Sliding algorithm (P. Z. Wang 2017)\cite{Wang2017}.

Inspired by Prof. Peizhang Wang’s original ideas, 
this paper will employee the Hat projection from statistical learning, it is easier and clearer than Wang's projection calculation in grading falling algorithm and the null projection in Lui's gravity sliding algorithm. In this paper, we will use the Hat matrix to calculate multiple projections, which is more precise than the methods mentioned before.

We introduce the Cone-cutting Theory as well as the past research on Sliding Gradient algorithm in Section \ref{sec:The Cone-cutting Theory and GSA}; We propose the linear adjusting programming in Section \ref{sec:Linear Adjusting Programming}; The Hat  projection is introduced in Section \ref{sec:Hat Projection}; The algorithm of linear adjusting programming and an example and the comparison with the Simplex, Episode, Interior Point method are given in Section \ref{sec:Algorithm of Linear Adjusting Programing}; Conclusions are put in Section \ref{sec:Conclusions}.

\section{Linear Programming} \label{sec:Linear Programming}

Linear programming (LP) (also known as linear optimization) is a technique to accomplish the best outcome (such as maximum revenue or lowest loss) in a mathematical model.

More formally, linear programming is a method for the optimization of a linear objective function, subject to linear equality and linear inequality constraints. Its feasible region is a convex polytope, which is a set defined as the intersection of finitely many half spaces, each of which is defined by a linear inequality (Dimitris Bertsimas, John N. Tsitsiklis 1997)\cite{Dimitris1997}. Its objective function is a real-valued affine (linear) function defined on this polyhedron. A linear programming algorithm finds a point in the polyhedron where this function has the smallest (or largest) value if such a point exists (Dimitris Bertsimas, John N. Tsitsiklis 1997)\cite{Dimitris1997}.

Linear programs are problems that can be expressed in canonical form as (Dimitris Bertsimas, John N. Tsitsiklis 1997)\cite{Dimitris1997}	
\renewcommand{\theequation}{\thesection.\arabic{equation}}
\setcounter{equation}{0}
\begin{equation}
\begin{aligned}
maximize\;\;\emph{c}^{\emph{T}} \emph{x} \\
\; subject\; to\;\;\emph{Ax}\leq \emph{b}  \\
\;and\;\;\emph{x} \geq 0 
\label{diseqn}
\end{aligned}
\end{equation}
where \emph{x} represents the vector of variables (to be determined), \emph{c} and \emph{b} are vectors of (known) coefficients, \emph{A} is a (known) matrix of coefficients, and $(\cdot)^{\emph{T}} $ is the matrix transpose. The expression to be maximized or minimized is called the objective function ($\emph{c}^{\emph{T}} \emph{x}$ in this case). The inequalities $\emph{Ax}\leq \emph{b}$ and $\emph{x} \geq 0$ are the constraints which specify a convex polytope over which the objective function is to be optimized. In this context, two vectors are comparable when they have the same dimensions. If every entry in the first is less-than or equal-to the corresponding entry in the second then we can say the first vector is less-than or equal-to the second vector.

Standard form is the usual and most intuitive form of describing a linear programming problem. It consists of the following three parts:

A linear function to be maximized \emph{e.g.}
\begin{equation}
\begin{aligned}
\emph{f}(\emph{x}_{1},\emph{x}_{2})=\emph{c}_{1}\emph{x}_{1}+\emph{c}_{2}x_{2}
\label{diseqn}
\end{aligned}
\end{equation}

Problem constraints of the following form \emph{e.g.}
\begin{equation}
\begin{aligned}
\emph{a}_{11}\emph{x}_{1}+\emph{a}_{12}\emph{x}_{2}\leq \emph{b}_{1}\\
\emph{a}_{21}\emph{x}_{1}+\emph{a}_{22}\emph{x}_{2}\leq \emph{b}_{2}\\
\emph{a}_{31}\emph{x}_{1}+\emph{a}_{32}\emph{x}_{2}\leq \emph{b}_{3}\\
\end{aligned}
\label{diseqn}
\end{equation}

Non-negative variables \emph{e.g.}
\begin{equation}
\begin{aligned}
\emph{x}_{1}\geq 0\\\emph{x}_{2}\geq 0
\end{aligned}
\label{diseqn}
\end{equation}

The problem is usually expressed in matrix form, and then becomes:
\begin{equation}
\begin{aligned}
maximize \;\{\emph{c} ^{\emph{T} }\emph{x} \;|\;\emph{A}\emph{x} \leq \emph{b} \land \emph{x} \geq 0\}
\end{aligned}
\label{diseqn}
\end{equation}

Other forms, such as minimization problems, problems with constraints on alternative forms, as well as problems involving negative variables can always be rewritten into an equivalent problem in standard form.
There are several open problems in the theory of linear programming, the solution of which would represent fundamental breakthroughs in mathematics and potentially major advances in our ability to solve large-scale linear programs.

Does LP admit a strongly polynomial-time algorithm?

Does LP admit a strongly polynomial-time algorithm to find a strictly complementary solution?

Does LP admit a polynomial-time algorithm in the real number (unit cost) model of computation?

These closely related set of problems have been cited by Stephen Smale as among the 18 greatest unsolved problems of the 21st century. In Smale's words, the third version of the problem is ``the main unsolved problem of linear programming theory". While algorithms exist to solve linear programming in weakly polynomial time, such as the ellipsoid methods and interior-point techniques, no algorithms have yet been found allowing strongly polynomial-time performance in the number of constraints and the number of variables. The development of such algorithms would be of great theoretical interest, and perhaps allow practical gains in solving large LPs as well.
Although the Hirsch conjecture was recently disproved for higher dimensions, it still leaves the following questions open (M. Hassan, M. Rehmani, and J. Chen 2020)\cite{Chen2020} (P. Wang. 2019)\cite{Wang2019}.

Are there pivot rules which lead to polynomial-time simplex variants?

Do all polytopal graphs have polynomially bounded diameter?

These questions relate to the performance analysis and development of simplex-like methods.

\section{The Cone-cutting Theory and Gravity Sliding algorithm} \label{sec:The Cone-cutting Theory and GSA}
In a \textit{m}-dimension space $R^{m}$, a hyperplane $y^{T}\tau={c}$ cuts $R^{m}$ into two half spaces. Here $\tau$ is a normal vector of the hyperplane and $c$ is a constant. We denote the positive half space $\{y|y^{T}\tau\geq{c}\}$ the accept zone of the hyperplane and the negative half space where $\{y|y^{T}\tau<{c}\}$ its reject zone. We define a facet as follows:\\
\textbf{Definition 3.1} A facet $\alpha:(\tau,c)$ is a hyperplane where $\tau$ is a normal vector pointing to the accept zone of the half space and $c$ is a constant.

When there are $m$ facets in $R^{m}$ and $\{\tau_1,\tau_2,…,\tau_{m}\}$ are linear independent, this set of linear equations has a unique solution which is a point $V$ in $R^{m}$. Geometrically, $V$ is the vertex of the cone formed by facets $\{\alpha_1,\alpha_2,…,\alpha_{m}\}$ . We now give a formal definition of a cone:\\
\textbf{Definition 3.2} Given $m$ hyperplanes in $R^{m}$, with rank $r(\alpha_1,…,\alpha_{m})=m$ and intersection $V$,
$C=C(V;\alpha_1,…,\alpha_{m})=\alpha_1 \cap,…,\cap \alpha_{m}$ is called a cone in $R^{m}$. The area $\{y|y^{T}\tau_{i}\geq c_{i} (i=1,2,..,m)\}$ is called the accept zone of $C$. The point $V$ is the vertex and $\alpha_{j}$ is the facet plane, or simply the facet of $C$.

A cone $C$ also has $m$ edge lines. They are formed by the intersection of $(m-1)$ facets. Hence, a cone can also be defined as follows:\\
\textbf{Definition 3.2*} Given $m$ rays $R_{j}=\{V+tr_{j}|0\leq t <+\infty\}(j=1,…,m)$ shooting from a point $V$ with rank $r(r_1,...r_{m})=m, C=C(V;r_1,…,r_{m})=m, C=C(V;r_1,…,r_{m})=c[R_1,…,R{m}]$, the convex closure of $m$ rays is called a cone in $R^{m}$. $R_{j}$ is the edge, $r_{j}$ the edge direction, and  $R_{j}^{+}=\{V+tr_{j}|-\infty<t<+\infty\}$ the edge line of the cone $C$.

The two definitions are equivalent, Furthermore, P.Z. Wang has observed that $R_{i}^{+}$ and $\alpha_{i}$ are opposite to each other for $i=1,…,m$. Edge-line $R_{i}^{+}$ is the intersection of all $C$-facets except $\alpha_{i}$, while the feasible part of facet $\alpha_{i}$ is bounded by all $C$-edges except $R_{i}^{+}$. This is the duality between facets and edges. Moreover, $r_{j}^{T}\tau_{i}=0$ ($i \neq j$) since $r_{j}$ lies on $\alpha_{i}$; and\\
\renewcommand{\theequation}{\thesection.\arabic{equation}}
\setcounter{equation}{0}
\begin{equation}
r_{i}^{T}\tau_{i}\geq 0 \; (i=1,...,m)
\label{diseqn}
\end{equation}
\subsection{Cone-cutting and Simplex Tableau}\label{subsec:Cone Cutting and Simplex Tableau}
Consider a linear programming (LP) problem and its dual:\\
\begin{equation}
\begin{aligned}
	(Primary) : max\{\emph{c}^{\emph{T} }\emph{x} \;|\; \emph{A}\emph{x} \leq \emph{b}; x\geq 0\};\\
	(Dual) : min\{y^{T}b \; | \; \emph{y}^{T}A \geq c; y\geq 0\};
	\end{aligned}
	\label{diseqn}
\end{equation}

The standard simplex tableau can be obtained by appending an $m\times m$ identity matrix $I_{m\times m}$ which represents the slack variables:\\
\begin{equation}
\begin{bmatrix}\centering
\alpha_{11} & \cdots & \alpha_{1n} & 1 & \cdots &0& b_{1}\\
\vdots & \ddots & \vdots & \vdots &\ddots &\vdots &\vdots\\
\alpha_{m1} &\cdots& \alpha_{mn} &0& \cdots&1& b_{m}\\
c_{1} &\cdots& c_{n} &0& \cdots&0& 0\\
\end{bmatrix}
\end{equation}

Under the dual LP formulation, the simplex tableau can be represented as a matrix of $(n+m)$ column vectors; each column vector is a hyperplane of the form $y^{T}\tau_{j} \geq c_{j}$, where
\begin{equation}
\tau_{j}=\left\{
\begin{array}{lr}
A_{j},\;\;\;1\leq j\leq n \\
I_{j-n},\;\;\;n+1\leq j\leq n+m 
\end{array};
\right. 
c_{j}=\left\{
\begin{array}{lr}
 c_{j}, \;\;1\leq j\leq n\\
0,\;\;n+1\leq j\leq n+m
\end{array}
\right.
\label{sec3.4}
\end{equation}

Using the facet notation, the simplex tableau can then be represented as a matrix of facets $\alpha_{j} : (\tau_{j},c_{j}), j=1,...,n+m$ as shown in \ref{sec3.4}.
\begin{align}
\begin{matrix}
\alpha_1 & \alpha_2&\cdots& \alpha_{n+m} \\
[\tau_1&\tau_2 & \cdots&\tau_{n+m}]\\
\cline{1-4}
[c_1&c_2&\cdots&c_{n+m}]\\
\end{matrix}\end{align}

P.Z. Wang et al. have developed the cone-cutting theory (P. Wang. 2011)\cite{Wang2011} (P. Wang. 2014)\cite{Wang2014} to solve the dual LP problem. Amazingly, P.Z. Wang shows that this algorithm produces exactly the same result as the original simplex algorithm (G. B. Dantzig. 1963)\cite{Dantzig1963} when $b>0$. Hence, the cone-cutting theory offers a geometric interpretation of the simplex method. More significantly, it inspires the authors to explore new LP solutions.
\subsection{The Gravity Sliding algorithm} \label{subsec:Gravity Sliding algorithm}
Expanding on the cone-cutting theory, the Gravity Sliding Algorithm (P. Z. Wang 2017)\cite{Wang2017} and its variant, the Sliding Gradient Algorithm have been developed to find the optimal solution of the LP problem. This algorithm assumes that the dual feasible region $\mathcal{D}$ is bounded and non-empty. Recall that $\mathcal{D}$ is a convex polyhedron formed by constraints ($\emph{y}^{T}A \geq c; y\geq 0$) or ($y^{T}\tau_{j} \geq c_{j}; 1\leq j\leq n+m$), and the optimal feasible point is at one of its vertices. Let $\Omega=\{V_{i}|V_{i}^{T}\tau_{j} \geq c_{j}; j=1,...,n+m\}$ be the set of feasible vertices. The dual LP problem can then be stated as: $min \{V_{i}^{T}b | V_{i}\in \Omega \}$. Hence, the optimal vertex $V^*$ is the lowest vertex of $\mathcal{D}$ viewed from the direction of b. Thus, from a feasible point $P_{0}$ on or inside $\mathcal{D}$, we can set the gradient vector $g_{0}=-b$ as the general descend direction, and descend inside $\mathcal{D}$ to reach $V^*$. The descending path cannot penetrate $\mathcal{D}$ or else it would go to the infeasible region. In other words, the facets that form the dual feasible region $\mathcal{D}$ may block the descending path and they are called the blocking facets. The gradient descend vector needs to change direction when it is blocked by the constraint facets. The main idea of the Sliding Gradient Algorithm is to determine the new gradient direction $g$ when it encounters blocking facets and to find the next stage point along $g$. P.Z. Wang (P. Z. Wang. 2017)\cite{Wang2017} and Lui (Lui. 2018)\cite{WL2018} gives a detailed discussion on this subject.

\section{Linear Adjusting Programming} \label{sec:Linear Adjusting Programming}

\textbf{Definition 4.1} Given an objective vector \emph{g}=\emph{c} in a prime factor space $\textit{X}=\textit{R}^{n}$ and a group of constraint hyperplanes $\{\beta_{\emph{i}}:(\eta _{\emph{i}},\textsl{x})=\emph{b}_{\emph{i}}\}(\emph{i}= 1...\emph{m})$, where \emph{c}, \textit{x} stand for column vectors; $(\eta _{\emph{i}}, \textit{x})$ stands for the inner product of $\eta _{i}$ and $\textsl{x}.$ The prime linear adjusted programming is denoted as follows: 
\renewcommand{\theequation}{\thesection.\arabic{equation}}
\setcounter{equation}{0}
\begin{equation}
\begin{aligned}(AP):  Up \{(\emph{{c}}, \emph{x})\, |\, \emph{A}\emph{x}\leq \emph{b; x}\geq 0\}.\end{aligned}
\label{diseqn}
\end{equation}
which aims to get an upper value of $(\emph{c, x})$ in the constraint area $\cap \{\beta _\emph{i}\,|\,{\emph{i}} = 1...\emph{m}\}$, where $ \beta_i=\{\textsl{x}\,|\,(\eta _\emph{i},  \textsl{x} )\leq \emph{c} _\emph{i}\}.$\\
\textbf{Definition 4.2}\, Given an objective vector \emph{g=b} in the dual factor space $\textit{Y}=\textit{R}^{\emph{m}}$ and a group of constraint hyperplanes\, $\alpha_{j}=\{\textsl{y}\,|\,(\tau_{\emph{j}},\textsl{y})\geq \emph{c}_{\emph{j}}\} (\emph{j}= 1...\emph{n})$, where \emph{b}, \textsl{y} stand for row vectors, the dual linear adjusted programming is denoted as follows:  
\begin{equation}
\begin{aligned}(AD):  Lw \{(\emph{b}, \textsl{y}) \,|\, \textsl{y}\emph{A}\geq \emph{c; y}\geq 0\};\end{aligned}
\label{diseqn}
\end{equation}
which aims to get a lower value of \emph{(b, y)} under the constraints of $\cap \{ {\alpha _\emph{j}}\,|\,{\rm{\emph{j}}} = 1...\emph{n}\},$ where $\alpha _\emph{j}=\{\textsl{y}\,|\,(\tau _\emph{j}, \textsl{y} )\geq \emph{c} _\emph{j}\}$.

Linear adjustment programming is different from linear programming that the purpose of linear programming is to get the ultimate best target, while linear adjusting programming only makes a direct action in a short time. What resistance will come next? How does an agent orient himself? What further resistance may be encountered, and how should the direction be adjusted? If we can find the best, it is a blessing; if not, it is fine to find a local optimal solution. However, we request the next adjusting is given at least.

For convenience, we just state the linear adjusted programming problem in a dual form (AD).

There are two tasks that need to be overcome in linear adjusted programming: 1. Suppose that the position of the moving point is $\emph{y}=\emph{P}_\emph{t}$\;at time \emph{t}, given a direction \emph{d}, calculating the next stage point $\emph{P}_{\emph{t}+\emph{1}} $\, where one or several constraint hyperplanes block the moving line. 2. Calculating the projection of \emph{d} in a subspace.  

To perform the first task, we have the following proposition:
The next stage point can be calculated as follows:
\begin{equation}
\begin{split}
\textit{t}_\textit{j:}&=(\textit{c}_\textit{j}-(\tau_\textit{j},\textit{P}))/(\tau_\textit{j},\textit{d})\;\;(\textit{j} = 1,\ldots ,\textit{n}); \\\;\textit{j}^*:&=Argmin_\textit{j}\{ \textit{t}_\textit{j}\,|\,\textit{t}_\textit{j} > 0;\textit{j} = 1,...,\textit{n}\};\\ \textit{P}_{\textit{t}+1}:&=\textit{P}_\textit{t}+ \textit{t}_{\textit{j}^*}\textit{d}.
\end{split}
\label{equation4.3}
\end{equation}

The coordinate of a point on the ray moving ray starting from  $\emph{P}_{\emph{t}}$ is $\emph{y}=\emph{P}_{\emph{t}}+\emph{td} \,(\emph{t}>0)$, where \emph{t} is a parameter. If a constraint hyperplane ${{\alpha}_\emph{j}}:({\tau_\emph{j}},\emph{y}){\rm{ }} = {\emph{c}_\emph{j}}$, does not pass through the stage point $\emph{P}_\emph{t}$, then it meets the ray if and only if
\begin{equation}
\begin{aligned}
({\tau_\emph{j}},{\emph{P}_\emph{t}} + \emph{td})= {\emph{c}_\emph{j}}\;(\emph{t} > 0).
\end{aligned}
\label{diseqn}
\end{equation}

When$({\tau_\emph{j}},\emph{d}) \ne 0$, we can calculate the parameter \emph{t} for $\alpha_{\emph{j}}:$
\begin{equation}
\begin{aligned}
\emph{t}_\emph{j} = ({\emph{c}_\emph{j}} - ({\tau_\emph{j}},\emph{P}))/(\tau_\emph{j},\emph{d}).
\end{aligned}
\label{diseqn}
\end{equation}

Then select the first one who blocks the moving, let \emph{j}* be the index of it, we have that
\begin{equation}
\begin{aligned}
\textit{j}^* = Argmin_{\emph{j}}\{ \emph{t}_\emph{j}\;|\;{\emph{t}_\emph{j}} > 0;\emph{j} = 1,{\rm{ }}...,\emph{n}\} .
\end{aligned}
\label{diseqn}
\end{equation}
Therefore, \ref{equation4.3} holds.

\section{Hat Projection} \label{sec:Hat Projection}

Projection is an important concept in linear distance spaces. A projection is a mapping $\emph{P}:\emph{X}\rightarrow \emph{Y}$,\, where \emph{Y} is a subspace of linear distance space \emph{X}. For any point \emph{x} in \emph{X}, there is one and only one point \emph{p}(\emph{x}) in \emph{Y}, such that
\renewcommand{\theequation}{\thesection.\arabic{equation}}
\setcounter{equation}{0}
\begin{equation}
\begin{aligned}
(\forall \emph{y} \in \emph{Y})\;\emph{d}(\emph{x},\emph{p}\left( \emph{x} \right))) \le \emph{d}(\emph{x},\emph{y})
\end{aligned}
\label{quations4.1}
\end{equation}
where \emph{d} is a distance defined in \emph{X} (in subspace \emph{Y} also). Projection has following basic properties:
\begin{itemize}
	\item 
	\textbf{Idempotent Law:}
	\begin{equation}
	\begin{aligned}
	(\forall \emph{x} \in \emph{X})\;\emph{p}(\emph{p}(\emph{x})){\rm{ }} = \emph{p}(\emph{x});{\rm{ }}\;\end{aligned}
	\label{Law4.2}
	\end{equation}
	
	Projection obeys idempotent law, all image \emph{p}(\emph{x}) is its fixed points.
	\item 
	\textbf{Transitivity:}
	
	If $ \emph{Y}\subset \emph{A}\subset \emph{X}$, then\begin{equation}
	\begin{aligned}{ \downarrow ^\emph{A}}_\emph{Y}({ \downarrow _\emph{A}}\emph{X}){\rm{ }} = { \downarrow _\emph{Y}} \emph{X};{\rm{ }}\;\;\end{aligned}
	\label{diseqn}
	\end{equation}
	
	Projection relays in decrease subspaces. Where $\downarrow ^\emph{A}_\emph{Y} $ stands for projection $\emph{p}:{\rm{ }}\emph{A}\rightarrow \emph{Y}.\;{ \downarrow _\emph{Y}}$ stands for projection $\emph{P}:\emph{X}\rightarrow \emph{Y}$. We will use the symbol later.
	\item 
	\textbf{Complementarity:}
	
	Let $\emph{A}^ \bot $ be the complementary subspace ${\emph{A}^ \bot }\;$ in\, \emph{X}, then
	\begin{equation}
	\begin{aligned}
	(\forall \emph{x}\; \in \emph{X})\;\emph{x} = { \downarrow _\emph{Y}}\emph{x}\; + {\downarrow _\emph{Y}}_\bot \emph{x}\;
	\end{aligned}
	\label{equations4.4}
	\end{equation}
	
	Let $\sigma = {\rm{ }}\{ {\alpha_\emph{j}}:{\rm{ }}({\tau_\emph{j}},\emph{y}) = 0\,|\,\textit{j = 1,{\rm{ }} \ldots ,k}\}  = \{ {\alpha_{1,{\rm{ }} \ldots ,}}{\alpha_\textit{k}}\} $ be the set of planes blocking the way at a stage point with coefficient matrix $\textit{A}=\{ {\tau_1},{\rm{ }} \ldots ,{\tau_\textit{k}}\} .{\rm{ }}$ For the second task, we need to do projection calculation. Denote that
	\begin{equation}
	\begin{aligned}\emph{A} = \;\{ \emph{y} \,|\,({\tau_\emph{j}}, \emph{y}) = 0,{\rm{ }}\;\emph{j} = 1,{\rm{ }} \ldots ,\emph{k}\},
	\end{aligned}
	\label{diseqn}
	\end{equation}
	which is the subspace paralleling to the intersection of $\sigma$-blockers, the dimension of \textit{A} is \textit{m}-\textit{k}. It is obvious that the complementary subspace consists of all linear combinations of $\{ {\tau_1},\ldots ,{\tau_\textit{m}}\}, i.e.,\;\textit{A}{^ \bot } = \{ \textit{y}\,|\,\textit{y} = {\lambda _1}{\tau_1} +  \ldots  + {\lambda _\textit{m}}{\tau_\textit{m}}\} .$
	
	How does the calculation of the projection of a vector \textit{g} in a subspace work when \textit{k}$\textgreater$1?

If $\textit{A} = \textit{A}_{\textit{m}\times\textit{k}}$ is a full rank matrix with \textit{m}$\geq$\textit{k}, then it is obvious that $\textit{B}=\textit{A}^{\textit{T}}\textit{A}$ must be symmetric and reversible.
\end{itemize}
\textbf{Definition B}  Denote $\textit{H=AB}^{-1}\textit{A}^{\textit{T}}$, which is called a Hat matrix with respect to \textit{A}.
\begin{theorem}
	Hat matrix is projective, i.e., for any $\textit{y}\in{\textit{R}^\textit{m}}$, we have that $\textit{HHy=Hy}$.
\end{theorem}
\begin{proof}
	We have that $\emph{HHy} = \emph{A}{\emph{B}^{ - 1}}{\emph{A}^\emph{T}}\emph{A}{\emph{B}^{ - 1}}{\emph{A}^\emph{T}}\emph{y} = \emph{A}{\emph{B}^{ - 1}}\emph{B}{\emph{B}^{ - 1}}{\emph{A}^\emph{T}}\emph{y} = \emph{A}{\emph{B}^{ - 1}}{\emph{A}^\emph{T}}\emph{y} = \emph{Hy}.$
\end{proof}

As a transformation, a projective matrix is a mapping $\textit{H: X}\to \textit{Y}$, which satisfies the idempotent Law \ref{Law4.2}, and is also called the projective matrix. \textit{Y} is the image of \textit{X}, consists of all fixed points of \textit{H}, called the stable subspace with respect to \textit{H}. 
\begin{theorem} 
The stable subspace of Hat matrix $\textit{H}$ is $\textit{A}^{\bot}$, i.e., $\textit{Y}=\textit{A}^{\bot}.$
\end{theorem}

Set $\textit{C=HA}.$ For $\textit{j} = 1,{\rm{ }} \ldots ,{\rm{ }}\textit{m}$, we have that $\textit{C}{._\textit{j}} = {({\textit{H}_1}., \ldots ,\;{\textit{H}_\textit{m}}.)^\textit{T}}\textit{A}{._\textit{j}}\;$. \\Since $ \textit{HA=AB}^{-1}\textit{A}^{\textit{T}}\textit{A=AB}^{-1}\textit{B=A}$, i.e., $\textit{HA=A}$, we have that $\textit{C=A} $ and then $C{._j} = A{._j}$, so that $\textit{A}{._\textit{j}} = \textit{C}{._\textit{j}} = {({\textit{H}_1}., \ldots ,\;{\textit{H}_\textit{m}}.)^\textit{T}}\textit{A}{._\textit{j}} = \textit{HA}{._\textit{j}}$. Therefore, ${\tau_\textit{j}} = \textit{A}{._\textit{j}}$ is a fixed point of \textit{H}. It means that ${\textit{A}^ \bot } \subseteq \textit{Y}$.

If $0 \ne \textit{x} \in \textit{A},$ since that \,$(\textit{A}{._\textit{j}},\textit{x}){\rm{ }} = {\rm{ }}0$ holds for all \textit{j}=1, …, \textit{k}, we have that $\textit{A}^{\textit{T}}\textit{x}=0,$ and then $\textit{Hx=AB}^{-1}\textit{A}^{\textit{T}}\textit{x}=0$. It means that $\textit{x} \notin \textit{Y}$, so that $\textit{A} \cap \textit{Y}= \emptyset$. Since the dimension of \textit{A} is \textit{m-k}, the dimension of \textit{Y}. could not be larger than \textit{k}. Since the dimension of $\textit{A}{^\bot }$ is \textit{k}, while ${\textit{A}^ \bot } \cap \textit{Y}{\rm{ }} = \textit{A}{^ \bot }$, if $\textit{A}{^\bot } \ne \textit{Y}$, then the dimension of \textit{Y} must be larger than \textit{k}, this is a contradiction, so that $\textit{A}{^\bot } = \textit{Y}.$
\begin{theorem}The projection of \emph{g} in \textit{A} is that:
	\begin{equation}
	\begin{aligned}
	g\downarrow _{A}= g - Hg
	\end{aligned}
	\label{diseqn}
	\end{equation}
\end{theorem}

Since \textit{Y} is the direct sum of $\textit{A}{^ \bot }$  and \textit{A}, according to \ref{equations4.4}, for any $\textit{g} \in \textit{Y}$, we have that
\begin{equation}
\begin{aligned}\emph{g} = \emph{g}{ \downarrow _{\emph{A}{\rm{ }}\;}} + \;\emph{g}{ \downarrow _A}_ \bot 
\end{aligned}
\label{diseqn}
\end{equation}

Since $\textit{A}{^ \bot }$ is the stable subspace of \textit{H}, we have that $\textit{g}{ \downarrow _\textit{A}}_ \bot $=\textit{Hg}, \ref{quations4.1} is held.

So far, we have got the formula of projection calculation. Is it consistent with classical formula of projection shown in \ref{quations4.1} of the introduction? Yes! Suppose that $\textit{A} = \textit{A}_{\textit{n}\times{1}}= \tau = {({\tau_1},\ldots ,{\tau_\textit{m}})^\textit{T}}$, we have that $\textit{B} = {\textit{A}^\textit{T}}\textit{A} = (\tau,\tau)$, and ${\textit{B}^{ - 1}} = 1/(\tau,\tau)$, then $\;\textit{H} = \textit{A}{\textit{B}^{ - 1}}{\textit{A}^\textit{T}} = \textit{A}{\textit{A}^\textit{T}}/(\tau,\tau)$, where $\textit{AA}^{\textit{T}}=\textit{R}$ is a \textit{m}$\times$\textit{m} matrix with elements $\textit{r}_{\textit{ij}} = \tau_{\textit{i}}\tau_{\textit{j}}$. We have that:
\begin{equation}
\begin{split}
\textit{R}\textit{g}=(\tau_1(\tau,\textit{g}),\ldots ,\tau_\textit{m}(\tau,\emph{g}))^\textit{T}=(\tau_1, \ldots ,\tau_\textit{m})^\textit{T}(\tau,\emph{g})=(\tau,\emph{g})\tau;\\
\textit{Hg}=[\textit{A}\textit{A}^\textit{T}/(\tau,\tau)]\textit{g} =\textit{A}\textit{A}^\textit{T}\textit{g}\;/(\tau,\tau)=\textit{Rg}\;/(\tau,\tau)=\tau(\textit{g},\tau)/(\tau,\tau);\\
\textit{g}\downarrow_\alpha=\textit{g}-\textit{Hg=g}-\tau(\textit{g},\tau)/(\tau,\tau).
\end{split}
\end{equation}

It is coincident with \ref{AAA} of the introduction.

\section{Algorithm of Linear Adjusting Programming} \label{sec:Algorithm of Linear Adjusting Programing}

Suppose that the feasible region is not empty and a point  \emph{P}$_{0}$ in \textit{D} is given.
\renewcommand{\theequation}{\thesection.\arabic{equation}}
\setcounter{equation}{0}
\begin{equation}
\begin{aligned}
\textit{s}:=0; \textit{d}:=\textit{g}; \sigma:=empty.
\end{aligned}
\label{diseqn}
\end{equation}
\textbf{Step 1} Determination of the next stage point $\textit{P}_\textit{{s+1}} $ by means of \ref{equation3.3} of the linear adjusting programming;
\begin{equation}
\begin{aligned}
{\textit{t}_\textit{j}:{\rm{ }} = {[\textit{c}_\textit{j}} - ({\tau_\textit{j}},{\textit{P}_\textit{s}})]/({\tau_\textit{j}},\textit{d}){\rm{ }}\;(\textit{j} = 1, \ldots ,\textit{m + n};{\rm{ }}({\tau_\textit{j}},\textit{d}) \ne 0);}\\
{\textit{j}^* = Argmin_\textit{j}}\{ {\textit{t}_\textit{j}}\,|\;{\textit{t}_\textit{j}} > 0\};\\
{\textit{P}_{\textit{s} + 1}: = \textit{P}_\textit{s}} + {\textit{t}_{\textit{j}^*}}\textit{d}; Stop;\\
Or, \sigma: = \sigma + \{ {\alpha_\textit{j}}\,|\,\textit{j} =\textit{j}^*\}  = \{ {\alpha_{\;\left( 1 \right)}}, \ldots,\alpha_{\textit{k}}\} ; go\; to\; Step \;2.\\
\end{aligned}
\label{diseqn}
\end{equation}
\textbf{Step 2} Calculating Hat projection: 
\begin{equation}
\begin{aligned}
{\textit{A}: = \textit{s}\; = \{ {\alpha_{\;\left( 1 \right)}}, \ldots ,\;{\alpha_{\;(}}{{_\textit{k}}_)}\} ;}\\
{\textit{d}{ \downarrow _\textit{A}} = \textit{d} - \textit{Hd};}\\
{\textit{s}: = \textit{s + 1}; go\; back\; Step\; 1.}
\end{aligned}
\label{diseqn}
\end{equation}
\textbf{Example:}  \\
\textbf{Given (AD)}: $Lw \{ (\textit{b,y})|\textit{yA}\geq\textit{c};\textit{y}\geq 0\}$

Where constraint vectors are as follows: ${\tau_1}={\left({2,{\rm{ }}0,{\rm{ }}0,{\rm{ }}1,{\rm{ }}1} \right)^\textit{T}}$;${\textit{c}_1} = 1;{\tau_2} = {\left( { - 1,{\rm{ }}1,{\rm{ }}2,{\rm{ }}1,{\rm{ }}0} \right)^\textit{T}}$;\\${\textit{c}_2} = 1;{\tau_3} = {\left( {0,{\rm{ }}1,{\rm{ }}0,{\rm{ }}1,{\rm{ }}1} \right)^\textit{T}}$;${\textit{c}_3} = 2;{\tau_4} = {\left( { - 1,{\rm{ }}1,{\rm{ }}1,{\rm{ }}0,{\rm{ }}0} \right)^\textit{T}};{\textit{c}_4} = 3;{\tau_5} = {\left( {1,{\rm{ }}1,{\rm{ }}1,{\rm{ }}1,{\rm{ }}1} \right)^\textit{T}}$;${\textit{c}_5} = 4$;\\${\tau_6} = {\left( {1,{\rm{ }}0,{\rm{ }}0,{\rm{ }}0,{\rm{ }}0} \right)^\textit{T}}$;${\textit{c}_6} = 0$;${\tau_7} = {\left( {0,{\rm{ }}1,{\rm{ }}0,{\rm{ }}0,{\rm{ }}0} \right)^\textit{T}};{\textit{c}_7} = 0;{\tau_8} = {\left( {0,{\rm{ }}0,{\rm{ }}1,{\rm{ }}0,{\rm{ }}0} \right)^\textit{T}};{\textit{c}_8} = 0;{\tau_9} = {\left( {0,{\rm{ }}0,{\rm{ }}0,{\rm{ }}1,{\rm{ }}0} \right)^\textit{T}};\\{\textit{c}_9} = 0;{\tau_{10}} = {\left( {0,{\rm{ }}0,{\rm{ }}0,{\rm{ }}0,{\rm{ }}1} \right)^\textit{T}};{\textit{c}_{10}} = 0;\;$ The dual objective vector is $\textit{g}=-\textit{O}=(-4, -1, -4, -6, -2)$. And the current point $\textit{P}_{0} = (7, 4, 7, 6, 5)$, so let’s do linear adjusting programming in some steps. 
\begin{equation}
\begin{split}
\textit{s}: =0; \textit{d}: =\textit{g};  \textit{P}:=\textit{P}_{0}; \sigma:=empty;
\end{split}
\end{equation}
\textbf{Step 1} Calculating the next stage point:
\begin{equation}
\begin{split}
\textit{t}_{j}: = [\textit{c}_{j}- (\tau_{j},\textit{P}_{s})]/(\tau_{j},\textit{d})\,(\textit{j}=1,\ldots ,10);\\
\textit{j}^* = Argmin_{j}\{\textit{t}_{j}\;|\;\textit{t}_{j}>0\}= Argmin_{j}\{\textit{t}_{4} = 1,\;\textit{t}_{9} = 1\}  = \left\{ 4,9\right\};\\
\sigma: =\sigma + \left\{4,9\right\} = \left\{4,9\right\};\\
\textit{P}_{1}: = \textit{P}_{0} +\textit{t}_{j}^*\textit{d} = (7,4,7,6,5) + (-4,-1,-4,-6,-2) = (3,3,3,0,3).
\end{split}
\end{equation}
\textbf{Step 2} Calculating Hat projection of \textit{g} in the subspace of $\{\tau_{\textit{j}}\,|\, \textit{j} \in \sigma \}$ ;\\
\setlength{\arraycolsep}{2.0pt}
$\begin{array}{c}
\textit{B}\!=\!{\textit{A}^\textit{T}}\textit{A}\!=\!\left( 
{\begin{array}{*{20}{c}}
{ - 1}&1&1&0&0\\
0&0&0&1&0
\end{array}} \right)\left( {\begin{array}{*{20}{c}}
{ - 1}&0\\
1&0\\
1&0\\
0&1\\
0&0
\end{array}} \right)\!=\!\left( {\begin{array}{*{20}{c}}
3&0\\
0&1
\end{array}} \right)\;\;\;\;\;

{\textit{B}^{ - 1}} = \left( {\begin{array}{*{20}{c}}
{1/3}&0\\
0&1
\end{array}} \right)\\
\textit{H = A}{\textit{B}^{ - 1}}{\textit{A}^\textit{T}} = \left( {\begin{array}{*{20}{c}}
{ - 1}&0\\
1&0\\
1&0\\
0&1\\
0&0
\end{array}} \right)\left({\begin{array}{*{20}{c}}
{1/3}&0\\
0&1
\end{array}} \right)\left({\begin{array}{*{20}{c}}
{ - 1}&1&1&0&0\\
0&0&0&1&0
\end{array}} \right)= \left({\begin{array}{*{20}{c}}
{1/3}&{ - 1/3}&{ - 1/3}&0&0\\
{ - 1/3}&{1/3}&{1/3}&0&0\\
{ - 1/3}&{1/3}&{1/3}&0&0\\
0&0&0&1&0\\
0&0&0&0&0
\end{array}} \right)
\end{array}$
\begin{equation}
\begin{split}
\textit{d}:=\;\textit{g}\downarrow _\textit{A} = \textit{g} - \textit{Hg} =\left(-13,-2,-11,0,-6\right)/3;\\
\textit{s}:=\textit{s} + 1 = 1;go\; back\; Step\; 1.
\end{split}
\end{equation}
\textbf{Step 1} Determining next stage point:
\begin{equation}
\begin{split}
\textit{j}^* = Argmin_\textit{j}\{\textit{t}_\textit{j}= [\textit{c}_\textit{j}-(\tau_\textit{j},\textit{P})]/(\tau_\textit{j},\textit{d})\,|\,\textit{t}_\textit{j}>0=Argmin_\textit{j}\{ \textit{t}_{6} = 3/13\}= 6;\\
\textit{P}_{2}:=\left(3,3,3,0,3\right)+3/13\left(-13,-2,-11,0,-6 \right)=\left(0,33/13,6/13,0,21/13\right);\\
\sigma :=\;\sigma \; \cup \{\textit{j}\,|\,\textit{j} = \textit{j}^*\} = \left\{4,9,6\right\}.
\end{split}
\end{equation}
\textbf{Step 2} Calculating Hat projection:
\begin{center}
\setlength{\arraycolsep}{1.8pt}
$\begin{array}{c}
	\textit{B}\!=\! \textit{A}{\textit{A}^\textit{T}}\!=\!\left( {\begin{array}{*{20}{c}}
			{-1}&1&1&0&0\\
			0&0&0&1&0\\
			1&0&0&0&0
	\end{array}} \right)\left({\begin{array}{*{20}{c}}
			{-1}&0&1\\
			1&0&0\\
			1&0&0\\
			0&1&0\\
			0&0&0
	\end{array}} \right)\!=\! \left( {\begin{array}{*{20}{c}}
			3&0&{-1}\\
			0&1&0\\
			{-1}&0&1
	\end{array}} \right)
\end{array}$

\setlength{\arraycolsep}{1.8pt}
$\begin{array}{c}
	(\textit{B},\textit{I}) = \left( {\begin{array}{*{20}{c}}
			3&0&{-1}&1&0&0\\
			0&1&0&0&1&0\\
			{-1}&0&1&0&0&1
	\end{array}} \right) \to \left( {\begin{array}{*{20}{c}}
			1&0&{-1/3}&{1/3}&0&0\\
			0&1&0&0&1&0\\
			0&0&{2/3}&{1/3}&0&1
	\end{array}} \right)  \to \left( {\begin{array}{*{20}{c}}
			1&0&0&{1/2}&0&{1/2}\\
			0&1&0&0&1&0\\
			0&0&1&{1/2}&0&{3/2}
	\end{array}} \right) = (\textit{I},{\textit{B}^{ - 1}})
\end{array}$
$\begin{array}{c}
	{\textit{B}^{-1}} = \left( {\begin{array}{*{20}{c}}
			{1/2}&0&{1/2}\\
			0&1&0\\
			{1/2}&0&{3/2}
	\end{array}} \right)
\end{array}$

\setlength{\arraycolsep}{1.8pt}
$\begin{array}{c}
	\textit{H} =\textit{A}{\textit{B}^{ - 1}}{\textit{A}^{ - \textit{T}}} \!=\!\left({\begin{array}{*{20}{c}}
			{-1}&0&1\\
			1&0&0\\
			1&0&0\\
			0&1&0\\
			0&0&0
	\end{array}}\right)\left( {\begin{array}{*{20}{c}}
			{1/2}&0&{1/2}\\
			0&1&0\\
			{1/2}&0&{3/2}
	\end{array}} \right)\left( {\begin{array}{*{20}{c}}
			{-1}&1&1&0&0\\
			0&0&0&1&0\\
			1&0&0&0&0
	\end{array}} \right) = \left( {\begin{array}{*{20}{c}}
			1&0&0&0&0\\
			0&{1/2}&{1/2}&0&0\\
			0&{1/2}&{1/2}&0&0\\
			0&0&0&1&0\\
			0&0&0&0&0
	\end{array}} \right)
\end{array}$
\end{center}
(Where the inverse matrix is calculated by elimination method.)
\begin{equation}
\begin{aligned}
{\textit{d}:{\rm{ }} = \;\textit{g}{ \downarrow _\textit{A}} = \textit{g} - \textit{Hg} = {\rm{ }}\left( {0,{\rm{ }}3/2,{\rm{ }} - 3/2,{\rm{ }}0,{\rm{ }} - 2} \right);\;}\\
{\textit{s}:{\rm{ }} = \textit{s} + 1{\rm{ }} = {\rm{ }}2;go\; back\; Step\; 1.}
\end{aligned}
\label{diseqn}
\end{equation}
\textbf{Step 1} Determining the next stage point:
\begin{equation}
\begin{split}
\textit{j}^*=Argmin_\textit{j}\{\textit{t}_\textit{j}= [\textit{c}_\textit{j}- (\tau_\textit{j},\textit{P}_{2})]/(\tau_\textit{j},\textit{d})|\;\textit{t}_\textit{j}>0\}=Argmin_\textit{j}\{\textit{t}_{1} = \textit{t}_{5} = \textit{t}_{8} = \textit{t}_{9} = 4/13\}  = \{ 1,5,8,9\}\\
{P}_{3}:=\left(0,33/13,6/13,0,21/13\right)+4/13\left(0,3/2,- 3/2,0,-2\right)=\left(0,3,0,0,1\right)
\end{split}
\end{equation}

That is enough, so we stop here. Indeed we have got the dual optimal point of the corresponding LP problem: ${\textit{y}_1}^* = 0,{\textit{y}_2}^* = 3,\;{\textit{y}_3}^* = 0,\;{\textit{y}_4}^* = 0,\;{\textit{y}_5}^* = 1.$

The following are attempts of a linear programming example by using different algorithms. Three algorithms are to be used: Simplex Method, Ellipsoid Method and one of interior point methods. \\
\textbf{Description of the example:}

\emph{R} is used to denote the set of real numbers. \emph{A} is a \emph{m}$\times$\emph{n} matrix;\emph{b} is a  column vector in $\emph{R}^\emph{m}$; \emph{c} is a column vector in $\emph{R}^\emph{n}$.
\begin{equation}
\begin{aligned}
(Primal)\;\emph{max} \{\emph{c}^\emph{T}\emph{x}:\emph{Ax}\leq \emph{b},\emph{x}\geq 0\} ;\\
(Dual)\;\emph{min} \{\emph{y}^\emph{T}\emph{b}:\emph{y}^\emph{T}\emph{A}\geq \emph{c}^\emph{T},\emph{y}\geq 0\}; 
\end{aligned}
\label{diseqn}
\end{equation}

Equivalently, the dual problem is $\emph{min}\{\emph{y}^\emph{T}\emph{b}:\emph{y}^\emph{T}[\emph{A I}]\geq [\emph{c}^\emph{T}0^\emph{T}]\}$;\\
\textbf{Given information:} 
\begin{equation}
\begin{aligned}
\setlength{\arraycolsep}{1.0pt}
\emph{A}\!=\!\begin{bmatrix}
2 & -1 & 0 & -1&1 \\
0 & 1 & 1& 1&1 \\
0& 2& 0 &1&1\\
1 & 1 &  1& 0&1\\
1&0&1&0&1\\
\end{bmatrix};\emph{b}\!=\!\begin{bmatrix}
4\\
1\\
4\\
6\\
2\\
\end{bmatrix};\emph{c}\!=\!\begin{bmatrix}
1\\
1\\
2\\
3\\
4\\
\end{bmatrix};
\textit{P}_0\!=\!\begin{bmatrix}
7\\
4\\
7\\
6\\
5\\\end{bmatrix};
\end{aligned}
\label{diseqn}
\end{equation}
\textbf{Simplex Method (full tableau)}:

Consider the primal problem with slack variable  \emph{s} $\geq 0$ such that:
\begin{equation}
\begin{aligned}
(Primal) \stackrel{\textit{equivalent}}\longleftrightarrow \emph{min} \{-\emph{c}^\emph{T}\emph{x}:\emph{Ax}+\emph{Is}=\emph{b},\emph{x}\geq 0,\emph{s}\geq 0\} ;
\end{aligned}
\label{diseqn}
\end{equation}

Pivoting rule used: Bland’s rule. (always choose the smallest index when deciding entering and leaving index)

We assume the index order is:

$\emph{x}_1 < \emph{x}_2 <\emph{x}_3 < \emph{x}_4 < \emph{x}_5 < \emph{s}_1 < \emph{s}_2 < \emph{s}_3 < \emph{s}_4 < \emph{s}_5$.

Notice that $\emph{b}\geq 0$, a natural basic feasible solution is given by \emph{x}=0, \emph{s}=\emph{b}, so, we have the initial tableau:\\
\textbf{Initial tableau:}
\setlength{\tabcolsep}{1.8mm}
\begin{table}[H]
\centering
\begin{tabular}[l]{|l|l|l|l|l|l|l|l|l|l|l|l|}
	\hline
	&$  $ & $\emph{x}^*_{1}$ & $\emph{x}_{2} $&$\emph{x}_{3} $&$\emph{x}_{4}$ & $\emph{x}_{5}$ & $\emph{s}_{1}$ & $\emph{s}_{2}$ & $\emph{s}_{3}$ & $\emph{s}_{4}$& $\emph{s}_{5}$ \\ \hline
	& 0 & -1 & -1 & -2 & -3 & -4 & 0 & 0 & 0 & 0 & 0 \\ \hline
	$\emph{s}^*_1$& 4 & 2 & -1 & 0 & -1 & 1 & 1 & 0 & 0 & 0 & 0 \\ \hline
	$\emph{s}_2$& 1 & 0 & 1 & 1 & 1 & 1 & 0 & 1 & 0 & 0 & 0 \\ \hline
	$\emph{s}_3$& 4 & 0 & 2 & 0 & 1 & 1 & 0 & 0 & 1 & 0 & 0 \\ \hline
	$\emph{s}_4$& 6 & 1 & 1 & 1 & 0 & 1 & 0 & 0 & 0 & 1 & 0 \\ \hline
	$\emph{s}_5$& 2 & 1 & 0 & 1 & 0 & 1 & 0 & 0 & 0 & 0 & 1 \\ \hline
\end{tabular}
\caption{initial tableau}
\label{table1}
\end{table}

\textbf{step 1:}
\setlength{\tabcolsep}{1.6mm}
\begin{table}[H]		
\centering
\begin{tabular}[l]{|l|l|l|l|l|l|l|l|l|l|l|l|}
	\hline
	&  & $\emph{x}_1$ & $\emph{x}^*_2$ &$\emph{x}_3$ &$\emph{x}_4$ & $\emph{x}_5$ & $\emph{s}_1$ & $\emph{s}_2$ & $\emph{s}_3$ & $\emph{s}_4$ & $\emph{s}_5$ \\ \hline
	& 2 & 0 & -1.5 & -2 & -3.5 & -3.5 & 0.5 & 0 & 0 & 0 & 0 \\ \hline
	$\emph{x}_1$& 2 & 1 & -0.5 & 0 & -0.5 & 0.5 & 0.5 & 0 & 0 & 0 & 0 \\ \hline
	$\emph{s}_2$& 1 & 0 & 1 & 1 & 1 & 1 & 0 & 1 & 0 & 0 & 0 \\ \hline
	$\emph{s}_3$& 4 & 0 & 2 & 0 & 1 & 1 & 0 & 0 & 1 & 0 & 0 \\ \hline
	$\emph{s}_4$& 4 & 0 & 1.5 & 1 & 0.5 & 0.5 & -0.5 & 0 & 0 & 1 & 0 \\ \hline
	$\emph{s}^*_5$& 0 & 0 & 0.5 & 1 & 0.5 & 0.5 & -0.5 & 0 & 0 & 0 & 1 \\ \hline
\end{tabular}
\caption{step 1}
\label{table2}
\end{table}
\textbf{step 2:}
\begin{table}[H]	
\centering
\begin{tabular}{|l|l|l|l|l|l|l|l|l|l|l|l|}
	\hline
	&  & $\emph{x}_1$ & $\emph{x}_2$ &$\emph{x}_3$ &$\emph{x}^*_4$ & $\emph{x}_5$ & $\emph{s}_1$ & $\emph{s}_2$ & $\emph{s}_3$ & $\emph{s}_4$ & $\emph{s}_5$ \\ \hline
	& 2 & 0 & 0 & 1 & -2 & -2 & -1 & 0 & 0 & 0 & 3 \\ \hline
	$\emph{x}_1$& 2 & 1 & 0 & 1 & 0 & 1 & 0 & 0 & 0 & 0 & 1 \\ \hline
	$\emph{s}_2$& 1 & 0 & 0 & -1 & 0 & 0 & 1 & 1 & 0 & 0 & -2 \\ \hline
	$\emph{s}_3$&4 & 0 & 0 & -4 & -1 & -1 & 2 & 0 & 1 & 0 & -4 \\ \hline
	$\emph{s}_4$& 4 & 0 & 0 & -2 & -1 & -1 & 1 & 0 & 0 & 1 & -3 \\ \hline
	$\emph{x}^*_2$& 0 & 0 & 1 & 2 & 1 & 1 & -1 & 0 & 0 & 0 & 2 \\ \hline
\end{tabular}
\caption{step 2}
\label{table3}
\end{table}
\textbf{step 3:}
\begin{table}[H]	
\centering
\begin{tabular}{|l|l|l|l|l|l|l|l|l|l|l|l|}
	\hline
	&  & $\emph{x}_1$ & $\emph{x}_2$ &$\emph{x}_3$ &$\emph{x}_4$ & $\emph{x}_5$ & $\emph{s}^*_1$ & $\emph{s}_2$ & $\emph{s}_3$ & $\emph{s}_4$ & $\emph{s}_5$ \\ \hline
	& 2 & 0 & 2 & 5 & 0 & 0 & -3 & 0 & 0 & 0 & 7 \\ \hline
	$\emph{x}_1$& 2 & 1 & 0 & 1 & 0 & 1 & 0 & 0 & 0 & 0 & 1 \\ \hline
	$\emph{s}^*_2$& 1 & 0 & 0 & -1 & 0 & 0 & 1 & 1 & 0 & 0 & -2 \\ \hline
	$\emph{s}_3$&4 & 0 & 1 & -2 & 0 & 0 & 1 & 0 & 1 & 0 & -2 \\ \hline
	$\emph{s}_4$& 4 & 0 & 1 & 0 & 0 & 0 & 0 & 0 & 0 & 1 & -1 \\ \hline
	$\emph{x}_4$& 0 & 0 & 1 & 2 & 1 & 1 & -1 & 0 & 0 & 0 & 2 \\ \hline
\end{tabular}
\caption{step 3}
\label{table4}
\end{table}
\textbf{step 4:}
\begin{table}[H]	
\centering
\begin{tabular}{|l|l|l|l|l|l|l|l|l|l|l|l|}
	\hline
	&  & $\emph{x}_1$ & $\emph{x}_2$ &$\emph{x}_3$ &$\emph{x}_4$ & $\emph{x}_5$ & $\emph{s}_1$ & $\emph{s}_2$ & $\emph{s}_3$ & $\emph{s}_4$ & $\emph{s}_5$ \\ \hline
	& 5 & 0 & 2 & 2 & 0 & 0 & 0 & 3 & 0 & 0 & 1 \\ \hline
	$\emph{x}_1$& 2 & 1 & 0 & 1 & 0 & 1 & 0 & 0 & 0 & 0 & 1 \\ \hline
	$\emph{s}_2$ & 1 & 0 & 0 & -1 & 0 & 0 & 1 & 1 & 0 & 0 & -2 \\ \hline
	$\emph{s}_3$ & 3 & 0 & 1 & -1 & 0 & 0 & 0 & -1 & 1 & 0 & 0 \\ \hline
	$\emph{s}_4$  & 4 & 0 & 1 & 0 & 0 & 0 & 0 & 0 & 0 & 1 & -1 \\ \hline
	$\emph{x}_4$ & 1 & 0 & 1 & 1 & 1 & 1 & 0 & 1 & 0 & 0 & 0 \\ \hline
\end{tabular}
\caption{step  4}
\label{table4}
\end{table}
Conclusion: we have $\emph{x}^*=\begin{bmatrix}
2&
0&
0&
1&
0\end{bmatrix}^\textit{T}$ as an optimal solution for the primal problem with the objective value $\textit{c}^\textit{T}\emph{x}=5$. Note that we only have four steps in total. \\	
\textbf{Ellipsoid Method:}

Recall that
\begin{equation}
\begin{aligned}
(Primal)\;\emph{max} \{\emph{c}^\emph{T}\emph{x}:\emph{Ax}\leq \emph{b},\emph{x}\geq 0\} ;\\
(Dual)\;\emph{min} \{\emph{y}^\emph{T}\emph{b}:\emph{y}^\emph{T}\emph{A}\geq \emph{c}^\emph{T},\emph{y}\geq 0\}; 
\end{aligned}
\label{diseqn}
\end{equation}
\begin{equation}
\begin{aligned}
\emph{A}= \begin{bmatrix}
2 & -1 & 0 & -1&1 \\
0 & 1 & 1& 1&1 \\
0& 2& 0 &1&1\\
1 & 1 &  1& 0&1\\
1&0&1&0&1\\
\end{bmatrix};\emph{b}=\begin{bmatrix}
4\\
1\\
4\\
6\\
2\\
\end{bmatrix};\emph{c}=\begin{bmatrix}
1\\
1\\
2\\
3\\
4\\
\end{bmatrix};
\end{aligned}
\label{diseqn}
\end{equation}

Ellipsoid method is used to find a feasible solution $\begin{bmatrix}
\emph{x}^*\\
\emph{y}^*\\
\end{bmatrix}$ for the following feasible region $\mathcal{D}$:\\
\begin{equation}
\begin{aligned}
\emph{c}^\emph{T}\emph{x}=\emph{y}^\emph{T}\emph{b},\emph{Ax}\leq \emph{b},\emph{x}\geq 0,\emph{y}^\emph{T}\emph{A}\geq \emph{c}^\emph{T},\emph{y}\geq 0
\end{aligned}
\label{diseqn}
\end{equation}

By the dual theory of linear programming, $\emph{x}^*$ is an optimal solution for the primal problem and  $\emph{y}^*$ is an optimal solution for the dual problem.\\
To represent  as a polyhedron:\\
\begin{equation}
\begin{aligned}
\mathcal{D}=\{\emph{A}_0\emph{p}\geq\emph{b}_0\}, \text{where}\; \emph{A}_0=\begin{bmatrix}
-\emph{A}&0\\
0&\emph{A}^T\\
\emph{I}&0\\
0&\emph{I}\\
\emph{c}^{T} & -\emph{b}\\
-\emph{c}^T & \emph{b}^T
\end{bmatrix},\emph{p}=\begin{bmatrix}
x\\
y\\
\end{bmatrix},\emph{b}_0=\begin{bmatrix}
-\emph{b}\\
c\\
0\\
0\\
0\\
0\\
\end{bmatrix};
\end{aligned}
\label{diseqn}
\end{equation}

Since all the entries are integers, we set the upper bound of all entries’ absolute values \emph{U}=6 by scanning the entries of \emph{A}, \emph{b}, \emph{c}. 
The dimension of variable is \emph{n}=10.
Note that $\emph{c}^\emph{T}\emph{x}=\emph{y}^\emph{T}\emph{b}$ ensures that the volume of $\mathcal{D}$ must be zero. In other words, there is no lower bound of the volume of $\mathcal{D}$. Therefore \emph{b} should be perturbed in order to give a stopping criterion for the ellipsoid method. 
Perturbation step:

$\frac{1}{\epsilon}=2(\emph{n+1})((\emph{n+1})\emph{U})^{\emph{n+1}}$\,$\approx$\,2.277225151082475$\times$$10^{21}$, and $\emph{b}_{0}$ is perturbed to be $\emph{b}_0-\epsilon$\emph{e} where \emph{e} is the column vector of all ones.    $\frac{1}{\epsilon}$ is the common denominator of $\emph{A}_0$ and $\emph{b}_0$, so after multiply  $\frac{1}{\epsilon}$ to each entry of  $\emph{A}_0$ and $\emph{b}_0$, we have the upper bound of absolute value of all the entries as $\tilde{\emph{U}}$$\approx$1.3663335090649485$\times$$10^{22}$. The estimated number of iterations is $\emph{O}(\emph{n}^{4}log(\emph{n}))$. Both the values for entries and the estimated number of iterations needed which is approximately $5\times10^5$ are too big for human to calculate. 

Conclusion: for this example the ellipsoid method is apparently not as practical as simplex method and gravity sliding algorithm. \\
\textbf{An interior point method: short-step affine scaling algorithm}:\\
Consider the dual problem:
\begin{equation}
\begin{aligned}
\setlength{\arraycolsep}{1.8pt}
\begin{array}{*{20}{l}}
(Dual) \emph{min} \{\emph{y}^\emph{T}\emph{b}:\emph{y}^\emph{T}\emph{A}\geq \emph{c}^\emph{T},\emph{y}\geq 0\}\stackrel{\textit{equivalent}}\longleftrightarrow \emph{min} \{\emph{b}^\emph{T}\emph{x}:[\emph{A}^T-\emph{I}]
\begin{bmatrix}
\emph{y}\\
\emph{s}\\
\end{bmatrix}=\emph{c},\emph{y}\geq 0, \emph{s}\geq 0\} ; 
\end{array}
\end{aligned}
\label{diseqn}
\end{equation}
\setlength{\arraycolsep}{1.8pt}
\begin{equation}
\begin{aligned}
\emph{A}= \begin{bmatrix}
2 & -1 & 0 & -1&1 \\
0 & 1 & 1& 1&1 \\
0& 2& 0 &1&1\\
1 & 1 &  1& 0&1\\
1&0&1&0&1\\
\end{bmatrix};\emph{b}=\begin{bmatrix}
4\\
1\\
4\\
6\\
2\\
\end{bmatrix};\emph{c}=\begin{bmatrix}
1\\
1\\
2\\
3\\
4\\
\end{bmatrix};
\textit{P}_0=\begin{bmatrix}
7\\
4\\
7\\
6\\
5\\\end{bmatrix};\\
\end{aligned}
\label{diseqn}
\end{equation}
For initialization: 

Instead, we solve ${min}\{\emph{b}^\emph{T}\emph{y}+\emph{M}\emph{s}_6:[\emph{A}^T-\emph{I}]\begin{bmatrix}
\emph{y}\\
\emph{s}\\
\end{bmatrix}+(\emph{c}-[\emph{A}^T-\emph{I}]\emph{e})\emph{s}_6=\emph{c},\emph{y}\geq 0, \emph{s}\geq 0\} $, where \emph{M} is a large number and \emph{e} is a column vector with all components equal to one. The initial point for the affine scaling algorithm to start with is $\emph{y}_{\emph{j}}=1$ for $\emph{j}=1,2,...,5$ and $\emph{s}_{\emph{j}} =1$ for $\emph{j}=1,2,...,6$. 
Inputs:
\begin{equation}
\begin{split}
\setlength{\arraycolsep}{1.8pt}
A_0\!=\!\left[
\begin{array}{ccccccccccc}
2 & 0 & 0 & 1 & 1 & -1 & 0 & 0 & 0 & 0 & -2  \\ 
-1 & 1 & 2 & 1 & 0 & 0 & -1 & 0 & 0 & 0 & -1 \\ 
0 & 1 & 0 & 1 & 1 & 0 & 0 & -1 & 0 & 0 & 0 \\ 
-1 & 1 & 1 & 0 & 0 & 0 & 0 & 0 & -1 & 0 & 3 \\
1 & 1 & 1 & 1 & 1 & 0 & 0 & 0 & 0 & -1 & 0 \\ 
\end{array}
\right]
\end{split}
\end{equation}
\begin{equation}
\begin{aligned}
c_0=\left[
\begin{array}{ccccccccccc}
4&1&4&6&2&0&0&0&0&0&10^4
\end{array}
\right]^\emph{T};\text{choose $\emph{M}=10^4$;}
\end{aligned}
\label{diseqn}
\end{equation}
\begin{equation}
\begin{aligned}
\emph{x}^0=[
\begin{array}{ccccccccccc}
1&1&1&1&1&1&1&1&1&1&1
\end{array}]^{T}\geq 0; \epsilon=0.01;\;\text{choose}\;  \beta=0.997.
\end{aligned}
\label{diseqn}
\end{equation}
\textbf{Implementation} (by using MATLAB command window):
\begin{itemize}
\item Iteration 1:

$\emph{X}_0=\emph{diag}({\emph{x}^0});
\emph{P}^0=(\emph{A}_0\emph{X}^{2}_{0}\emph{A}^\emph{T}_{0})^{-1}\emph{A}_0\emph{X}^{2}_{0}\emph{c}_{0}=\begin{bmatrix}
	-0.28\\
	-1.51\\
	-0.14\\
	2.21\\
	0.64\\
\end{bmatrix}\times 10^3;\\
\emph{r}^0=\emph{c}_0-\emph{A}^\emph{T}_0\emph{P}^0=
\begin{array}{ccccccccccc}
0.62&
-1.21&
0.16&
1.29&
-0.22&
-0.28&
-1.51&
-0.14&
2.21&
0.64&
1.29
\end{array}]^{T}\times 10^3;$

Optimality check: ask $\emph{r}^0>0$ and $\emph{e}^T\emph{X}_0\emph{r}^0<\epsilon$ ? Ans: No.

Unboundedness check: ask $-\emph{X}^{2}_0\emph{r}^0 \geq 0$ ? Ans: No.

Update of solution:

${x}^{1}={x}^{0}-\beta \frac{\emph{X}^{2}_{0}r^0}{\left\|\emph{X}_0\emph{r}^0\right\|}=[\begin{array}{ccccccccccc}0.83 & 1.33 & 0.95 & 0.64&1.06&1.08&1.42&1.04&0.39&0.82 & 0.64\end{array}]^{T}$

Objective value: $\emph{c}^{\emph{T}}_0\emph{x}^1=6.43\times10^3$.
\item Iteration 2:

$\emph{X}_1=\emph{diag}({\emph{x}^1});
\emph{P}^1=(\emph{A}_0\emph{X}^{2}_{1}\emph{A}^\emph{T}_{0})^{-1}\emph{A}_0\emph{X}^{2}_{1}\emph{c}_{0}=\begin{bmatrix}
0.32\\
-1.16\\
-0.48\\
2.51\\
-0.09\\
\end{bmatrix}\times 10^3;\\
\emph{r}^1=\emph{c}_0-\emph{A}^\emph{T}_0\emph{P}^1=\begin{array}{ccccccccccc}
0.80&
-0.78&
-0.10&
1.42&
0.25&
0.32&
-1.16&
-0.48&
2.51&
-0.09&
1.95
\end{array}]^{T}\times 10^3;$

Optimality check: ask $\emph{r}^1>0$ and $\emph{e}^T\emph{X}_1\emph{r}^1<\epsilon$ ? Ans: No.

Unboundedness check: ask $-\emph{X}^{2}_1\emph{r}^1 \geq 0$ ? Ans: No.

Update of solution:

$\emph{x}^2=\emph{x}^1-\beta\frac{\emph{X}^{2}_1\emph{r}^1}{\left\|\emph{X}_1\emph{r}^1\right\|}=[\begin{array}{ccccccccccc}
0.63&
1.82&
0.99&
0.44&
0.96&
0.95&
2.24&
1.22&
0.25&
0.84&
0.36
\end{array}]^\emph{T};$

Objective value: $\emph{c}^{\emph{T}}_0\emph{x}^2=3.60\times10^3.$
\item Iteration 3:

$\emph{X}_2=\emph{diag}({\emph{x}^2});$
$\emph{P}^2=(\emph{A}_0\emph{X}^{2}_{2}\emph{A}^\emph{T}_{0})^{-1}\emph{A}_0\emph{X}^{2}_{2}\emph{c}_{0}=\begin{bmatrix}
264\\
-502\\
-354\\
1657\\
-435\\
\end{bmatrix};$\\
$\emph{r}^2=\emph{c}_0-\emph{A}^\emph{T}_0\emph{P}^2=[\begin{array}{ccccccccccc}
1.07&
-0.37&
-0.21&
1.03&
0.53&
0.26&
-0.50&
-0.35&
1.66&
-0.44&
5.05
\end{array}]^{T}\times 10^3;$

Optimality check: ask $\emph{r}^2>0$ and $\emph{e}^T\emph{X}_2\emph{r}^2<\epsilon$ ? Ans: No.

Unboundedness check: ask $-\emph{X}^{2}_2\emph{r}^2 \geq 0$ ? Ans: No.

Update of solution:\\
$\emph{x}^3=\emph{x}^2-\beta\frac{\emph{X}^{2}_2\emph{r}^2}{\left\|\emph{X}_2\emph{r}^2\right\|}=[\begin{array}{ccccccccccc}
0.47&
2.30&
1.07&
0.36&
0.77&
0.85&
3.22&
1.43&
0.21&
0.96&
0.10
\end{array}]^{T};$

Objective value: $\emph{c}^{\emph{T}}_0\emph{x}^3=1059.$
\item Iteration 4:

$\emph{X}_3=\emph{diag}({\emph{x}^3});
\emph{P}^3=(\emph{A}_0\emph{X}^{2}_{3}\emph{A}^\emph{T}_{0})^{-1}\emph{A}_0\emph{X}^{2}_{3}\emph{c}_{0}=\begin{bmatrix}
9.53\\
-40.16\\
-46.58\\
217.5\\
-92.2\\
\end{bmatrix};$

$\emph{r}^3=\emph{c}_0-\emph{A}^\emph{T}_0\emph{P}^3=[\begin{array}{ccccccccccc}
254&
-37.1&
-40.0&
175.9&
131.2&
9.5&
-40.7&
-46.6&
21.8&
-92.2&
9.32 \times 10^3
\end{array}]^{T};$

Optimality check: ask $\emph{r}^3>0$ and $\emph{e}^T\emph{X}_3\emph{r}^3<\epsilon$ ? Ans: No.

Unboundedness check: ask $-\emph{X}^{2}_3\emph{r}^3 \geq 0$ ? Ans: No.

Update of solution:

$\emph{x}^4=\emph{x}^3-\beta\frac{\emph{X}^{2}_3\emph{r}^3}{\left\|\emph{X}_3\emph{r}^3\right\|}=[\begin{array}{ccccccccccc}
0.41&
2.49&
1.11&
0.34&
0.69&
0.85&
3.64&
1.52&
0.20&
1.05&
0.0039
\end{array}]^{T};$

Objective value: $\emph{c}^{\emph{T}}_0\emph{x}^4=51.01.$
\item Iteration 5:

$\emph{X}_4=\emph{diag}({\emph{x}^4});
\emph{P}^4=(\emph{A}_0\emph{X}^{2}_{4}\emph{A}^\emph{T}_{0})^{-1}\emph{A}_0\emph{X}^{2}_{4}\emph{c}_{0}=\begin{bmatrix}
1.38\\
0.16\\
-0.69\\
-1.09\\
0.81\\
\end{bmatrix};$

$\emph{r}^4=\emph{c}_0-\emph{A}^\emph{T}_0\emph{P}^4=[\begin{array}{ccccccccccc}
1.67&
-0.37&
1.78&
4.33&
0.49&
1.38&
0.16&
-0.69&
1.09&
0.81&
9.99965 \times 10^3
\end{array}]^{T};$

Optimality check: ask $\emph{r}^4>0$ and $\emph{e}^T\emph{X}_4\emph{r}^4<\epsilon$ ? Ans: No.

Unboundedness check: ask $-\emph{X}^{2}_4\emph{r}^4 \geq 0$ ? Ans: No.

Update of solution:

$\emph{x}^5=\emph{x}^4-\beta\frac{\emph{X}^{2}_4\emph{r}^4}{\left\|\emph{X}_4\emph{r}^4\right\|}=[\begin{array}{ccccccccccc}
0.40&
2.55&
1.06&
0.32&
0.69&
0.82&
3.58&
1.56&
0.20&
1.02&
0.00003
\end{array}]^{T};$

Objective value: $\emph{c}^{\emph{T}}_0\emph{x}^5=11.98.$
\item Iteration 6:

$\emph{X}_5=\emph{diag}({\emph{x}^5});
\emph{P}^5=(\emph{A}_0\emph{X}^{2}_{5}\emph{A}^\emph{T}_{0})^{-1}\emph{A}_0\emph{X}^{2}_{5}\emph{c}_{0}=\begin{bmatrix}
1.39\\
0.20\\
-0.57\\
0.73\\
0.93\\
\end{bmatrix};$

$\emph{r}^5=\emph{c}_0-\emph{A}^\emph{T}_0\emph{P}^5=[\begin{array}{ccccccccccc}
1.22&
-0.30&
1.93&
4.04&
0.25&
1.39&
0.20&
-0.57&
0.73&
0.93&
1.0001 \times 10^4
\end{array}]^{T};$

Optimality check: ask $\emph{r}^5>0$ and $\emph{e}^T\emph{X}_5\emph{r}^5<\epsilon$ ? Ans: No.

Unboundedness check: ask $-\emph{X}^{2}_5\emph{r}^5 \geq 0$ ? Ans: No.

Update of solution:

$\emph{x}^6=\emph{x}^5-\beta\frac{\emph{X}^{2}_5\emph{r}^5}{\left\|\emph{X}_5\emph{r}^5\right\|}=[\begin{array}{ccccccccccc}
0.34&
3.15&
0.39&
0.19&
0.65&
0.53&
2.77&
1.99&
0.19&
0.72&
0.00002
\end{array}]^{T};$

Objective value: $\emph{c}^{\emph{T}}_0\emph{x}^6=8.77.$
\item Iteration 7:

$\emph{X}_6=\emph{diag}({\emph{x}^6});
\emph{P}^6=(\emph{A}_0\emph{X}^{2}_{6}\emph{A}^\emph{T}_{0})^{-1}\emph{A}_0\emph{X}^{2}_{6}\emph{c}_{0}=\begin{bmatrix}
1.32\\
0.076\\
-0.041\\
0.27\\
0.73\\
\end{bmatrix};$

$\emph{r}^6=\emph{c}_0-\emph{A}^\emph{T}_0\emph{P}^6=[\begin{array}{ccccccccccc}
0.97&
-0.03&
2.85&
3.92&
-0.0075&
1.32&
0.076&
-0.041&
-0.27&
0.73&
1.0002 \times 10^4
\end{array}]^{T};$

Optimality check: ask $\emph{r}^6>0$ and $\emph{e}^T\emph{X}_6\emph{r}^6<\epsilon$ ? Ans: No.

Unboundedness check: ask $-\emph{X}^{2}_6\emph{r}^6 \geq 0$ ? Ans: No.

Update of solution:

$\emph{x}^7=\emph{x}^6-\beta\frac{\emph{X}^{2}_6\emph{r}^6}{\left\|\emph{X}_6\emph{r}^6\right\|}=[\begin{array}{ccccccccccc}
0.97&
3.33&
0.13&
0.106&
0.654&
0.31&
2.42&
2.09&
0.19&
0.50&
0.00002
\end{array}]^{T};$

Objective value: $\emph{c}^{\emph{T}}_0\emph{x}^7=7.10.$
\item Iteration 8:

$\emph{X}_7=\emph{diag}({\emph{x}^7});
\emph{P}^7=(\emph{A}_0\emph{X}^{2}_{7}\emph{A}^\emph{T}_{0})^{-1}\emph{A}_0\emph{X}^{2}_{7}\emph{c}_{0}=\begin{bmatrix}
1.45\\
0.018\\
0.0080\\
0.356\\
0.616\\
\end{bmatrix};$

$\emph{r}^7=\emph{c}_0-\emph{A}^\emph{T}_0\emph{P}^7=[\begin{array}{ccccccccccc}
0.85&
0.0022&
2.99&
3.90&
-0.078&
1.45&
0.018&
0.0080&
1.45&
0.018&
1.0002 \times 10^4
\end{array}]^{T};$

Optimality check: ask $\emph{r}^7>0$ and $\emph{e}^T\emph{X}_7\emph{r}^7<\epsilon$ ? Ans: No.

Unboundedness check: ask $-\emph{X}^{2}_7\emph{r}^7 \geq 0$ ? Ans: No.

Update of solution:

$\emph{x}^8=\emph{x}^7-\beta\frac{\emph{X}^{2}_7\emph{r}^7}{\left\|\emph{X}_7\emph{r}^7\right\|}=[\begin{array}{ccccccccccc}
0.20&
3.30&
0.071&
0.055&
0.69&
0.15&
2.30&
2.05&
0.17&
0.32&
0.155 \times 10^4
\end{array}]^{T};$

Objective value: $\emph{c}^{\emph{T}}_0\emph{x}^8=6.25.$
\item Iteration 9:\\
$\emph{X}_8=\emph{diag}({\emph{x}^8});
\emph{P}^8=(\emph{A}_0\emph{X}^{2}_{8}\emph{A}^\emph{T}_{0})^{-1}\emph{A}_0\emph{X}^{2}_{8}\emph{c}_{0}=\begin{bmatrix}
1.60\\
0.0079\\
0.00038\\
0.5029\\
0.4858\\
\end{bmatrix};$

$\emph{r}^8=\emph{c}_0-\emph{A}^\emph{T}_0\emph{P}^8=[\begin{array}{ccccccccccc}
0.82&
0.0030&
3.00&
3.90&
-0.089&
0.160&
0.0079&
0.00038&
0.50&
0.49&
1.0002 \times 10^4
\end{array}]^{T};$

Optimality check: ask $\emph{r}^8>0$ and $\emph{e}^T\emph{X}_8\emph{r}^8<\epsilon$ ? Ans: No.

Unboundedness check: ask $-\emph{X}^{2}_8\emph{r}^8 \geq 0$ ? Ans: No.

Update of solution:

$\emph{x}^9=\emph{x}^8-\beta\frac{\emph{X}^{2}_8\emph{r}^8}{\left\|\emph{X}_8\emph{r}^8\right\|}=[\begin{array}{ccccccccccc}
0.132&
3.23&
0.040&
0.0306&
0.781&
0.075&
2.21&
2.04&
0.14&
0.217&
0.00001
\end{array}]^{T};$

Objective value: $\emph{c}^{\emph{T}}_0\emph{x}^9=5.77.$
\item Iteration 10.

$\emph{X}_9=\emph{diag}({\emph{x}^9});
\emph{P}^9=(\emph{A}_0\emph{X}^{2}_{9}\emph{A}^\emph{T}_{0})^{-1}\emph{A}_0\emph{X}^{2}_{9}\emph{c}_{0}=\begin{bmatrix}
1.70\\
0.004\\
-0.00007\\
0.56\\
0.43\\
\end{bmatrix};$

$\emph{r}^9=\emph{c}_0-\emph{A}^\emph{T}_0\emph{P}^9=[\begin{array}{ccccccccccc}
0.91&
0.0002&
3.00&
3.95&
-0.04&
1.61&
0.004&
-0.00007&
0.56&
0.43&
1.00015\times10^4
\end{array}];$

Optimality check: ask $\emph{r}^9>0$ and $\emph{e}^T\emph{X}_9\emph{r}^9<\epsilon$ ? Ans: No.

Unboundedness check: ask $-\emph{X}^{2}_9\emph{r}^9 \geq 0$ ? Ans: No.

Update of solution:\\
$\emph{x}^{10}=\emph{x}^9-\beta\frac{\emph{X}^{2}_9\emph{r}^9}{\left\|\emph{X}_9\emph{r}^9\right\|}=[\begin{array}{ccccccccccc}
0.078&
3.16&
0.024&
0.018&
0.87&
0.044&
2.14&
2.05&
0.10&
0.15&
0.68\times 10^{-5} 
\end{array}];$

Objective value: $\emph{c}^{\emph{T}}_0\emph{x}^{10}=5.48.$    \\
\textbf{[PAUSED]}
\end{itemize}
Conclusion: At the last iteration, the value for \emph{r} is almost nonnegative. By observing the values of $\emph{x}^{k}$ in each \emph{k-th} iteration, the convergence of the first five entries of $\emph{x}^{k}$ to $\emph{y}^*=[
\begin{array}{ccccc}
0&
3&
0&
0&
1
\end{array}]^\emph{T}$ and the convergence of objective values to the optimal value 5 can be seen easily. However, much more iterations are certainly needed for $\emph{x}^k$ to get close enough to the optimal value. 

The simplex method is the most convenient, but it has been proven that it is not a polynomial-time algorithm, and it will spin down on the boundary of the feasible region. The advantage is that the method is simple and the disadvantage is that it does not have eyes. Existing simple algorithms are difficult to ensure that they do not descend into the slow orbit of a circle. Our method will avoid this and follow the path of falling body. We can see every step of our method, and they can not see it by themselves. If our descending path passes through the vertex, and the next projection direction also passes through the vertex, then Cone-cutting algorithm is used, and Cone-cutting algorithm is the same as the simplex method. Hence, we can take advantage of its convenience, when we need to leave the vertex, go back to our method.
\section{Conclusions} \label{sec:Conclusions}
The model of linear adjusting programming focusses on the adjustment technique in artificial intelligence. The core problem is the projective calculation, which extends the gradient method to more wide areas. Linear adjusting programming is a relaxation of linear programming; the algorithm LAP, indeed, has realized the idea of gradient method in linear programming, which raises new possible solution for the searching of a strongly polynomial-time algorithm in LP. In the future, a combination of factor space, gravity sliding method, and gradient flow method, namely factor space in triple differential, might be finely developed as the last straw to crack the Smale’s problem with LP.

\end{document}